\documentclass[12pt]{article}
\pdfoutput=1 

\usepackage[T1]{fontenc}
\usepackage{lmodern,amsmath,amsthm,amsfonts,amssymb,graphicx,float,microtype,thmtools,underscore,mathtools}
\usepackage[shortlabels]{enumitem}
\setlist[itemize]{topsep=0ex,itemsep=0ex,parsep=0.3ex}
\setlist[enumerate]{topsep=0ex,itemsep=0ex,parsep=0.3ex}
\usepackage[usenames,dvipsnames,svgnames,table]{xcolor}
\usepackage[unicode=true]{hyperref}
\hypersetup{ 
colorlinks,
linkcolor={blue!60!black},
citecolor={black},
urlcolor={blue!60!black},
pdftitle={Product structure of graphs with an excluded minor},
pdfauthor={Freddie~Illingworth, Alex~Scott, David~R.~Wood}} 
\usepackage[capitalise, compress, nameinlink, noabbrev]{cleveref}
\crefname{lem}{Lemma}{Lemmas}
\crefname{thm}{Theorem}{Theorems}
\crefname{cor}{Corollary}{Corollaries}
\newcommand{\defn}[1]{\textcolor{Maroon}{\emph{#1}}}
\usepackage[longnamesfirst,numbers,sort&compress]{natbib}
\makeatletter
\def\NAT@spacechar{~}
\makeatother
\usepackage[margin=30mm]{geometry}
\renewcommand{\baselinestretch}{1.1}
\allowdisplaybreaks

\DeclarePairedDelimiter{\floor}{\lfloor}{\rfloor}

\DeclarePairedDelimiter{\abs}{\lvert}{\rvert}
\DeclarePairedDelimiter{\set}{\{}{\}} 
\renewcommand{\epsilon}{\varepsilon}
\renewcommand{\emptyset}{\varnothing}

\renewcommand{\geq}{\geqslant}
\renewcommand{\leq}{\leqslant}
\DeclareMathOperator{\dist}{dist}
\DeclareMathOperator{\tw}{tw}
\DeclareMathOperator{\stw}{stw}
\DeclareMathOperator{\ltw}{ltw}
\DeclareMathOperator{\rtw}{rtw}
\newcommand{\JJ}{\mathcal{J}}
\newcommand{\BB}{\mathcal{B}}
\newcommand{\FF}{\mathcal{F}}
\newcommand{\GG}{\mathcal{G}}

\newcommand{\LL}{\mathcal{L}}
\newcommand{\NN}{\mathbb{N}}
\newcommand{\OO}{\mathcal{O}}
\newcommand{\WW}{\mathcal{W}}
\renewcommand{\thefootnote}{\fnsymbol{footnote}}
\theoremstyle{plain}
\newtheorem{thm}{Theorem}
\newtheorem{lem}[thm]{Lemma}
\newtheorem{cor}[thm]{Corollary}

\newtheorem{obs}[thm]{Observation}
\crefname{obs}{Observation}{Observations}
\theoremstyle{definition}

\begin{document}
\title{\bf\Large Product structure of graphs with an excluded minor}
\author{%
Freddie Illingworth\footnotemark[2] \qquad 
Alex Scott\footnotemark[2] \qquad David~R.~Wood\footnotemark[3]}
\date{}
\footnotetext[0]{\today. MSC classification: 05C83  	graph minors. }
\footnotetext[2]{Mathematical Institute, University of Oxford, United Kingdom (\texttt{\{illingworth,scott\}\allowbreak@maths\allowbreak.ox\allowbreak.ac\allowbreak.uk}). Research supported by EPSRC grant EP/V007327/1.}
\footnotetext[3]{\,School of Mathematics, Monash   University, Melbourne, Australia  (\texttt{david.wood@monash.edu}). Research supported by the Australian Research Council, and by a Visiting Research Fellowship of Merton College, University of Oxford.}
\sloppy
\maketitle

\begin{abstract}
This paper shows that $K_t$-minor-free (and $K_{s, t}$-minor-free) graphs $G$ are subgraphs of products of a tree-like graph $H$ (of bounded treewidth) and a complete graph $K_m$. Our results include optimal bounds on the treewidth of $H$ and optimal bounds (to within a constant factor) on $m$ in terms of the number of vertices of $G$ and the treewidth of $G$. These results follow from a more general theorem whose corollaries include a strengthening of the celebrated separator theorem of Alon, Seymour, and Thomas [\emph{J.~Amer.~Math.~Soc.}~1990] and the Planar Graph Product Structure Theorem of Dujmovi\'c~\emph{et~al.}~[\emph{J.~ACM}~2020]. \end{abstract}

\renewcommand{\thefootnote}{\arabic{footnote}}

\section{\large Introduction}\label{Intro}

Graph Product Structure Theory is a body of research which describes complicated graphs as subgraphs of products of simpler graphs.  Typically, the simpler graphs are tree-like, in the sense that they have bounded treewidth, which is the standard measure of how similar a graph is to a tree. (We postpone the definition of treewidth and other standard graph-theoretic concepts until \cref{sec:prelim}.)\ This area has recently received a lot of attention~\citep{DHHW,BDJMW22,DJMMUW20,UTW,DMW,HW21b,DHJLW21,BDHK,HJMW,UWY22,Wood22} with highlights including the Planar Graph Product Structure Theorem of \citet{DJMMUW20}; see \cref{PGPST} below. 


Our main contribution is a powerful general result, \cref{JJstMain}, that essentially converts a tree-decomposition of a graph excluding a particular minor into a product that inherits some of the properties of the decomposition. Its applications include a strengthening of the celebrated Alon-Seymour-Thomas separator theorem as well as the Planar Graph Product Structure Theorem.

Throughout the paper we work with strong products of graphs. The \defn{strong product} $A \boxtimes B$ of graphs $A$ and $B$ has vertex-set $V(A) \times V(B)$, where distinct vertices $(v,x), (w,y)$ are adjacent if $v = w$ and $xy \in E(B)$, or $x = y$ and $vw \in E(A)$, or $vw \in E(A)$ and $xy \in E(B)$. This paper focuses on products of the form $H \boxtimes K_m$ and $H \boxtimes P \boxtimes K_m$ where $H$ is a graph of bounded treewidth, $P$ is a path and $m$ is some function of the original graph. An alternative view of the product $H \boxtimes K_m$ is as a `blow-up' of the graph $H$, obtained by replacing each vertex of $H$ be a copy of the complete graph $K_m$ and each edge of $H$ by a copy of the complete bipartite graph $K_{m, m}$.

In one of the cornerstone results of Graph Minor Theory, \citet{AST90} proved that every $K_t$-minor-free graph has a balanced separator of size at most $t^{3/2} n^{1/2}$. In fact, they proved the following stronger result.\footnote{The balanced separator result follows from \cref{AST} and the separator lemma of \citet[(2.6)]{RS-II}.}

\begin{thm}[\citep{AST90}] 
\label{AST}
Every $n$-vertex $K_t$-minor-free graph $G$ has treewidth $\tw(G)< t^{3/2} n^{1/2}$. 
\end{thm}

Our first result is the following strengthening of \cref{AST} that describes $K_t$-minor-free graphs as blow-ups of simpler graphs, namely graphs with bounded treewidth.

\begin{thm}
\label{KtMinorFreeSqrt}
For any integer $t \geq 4$, every $n$-vertex $K_t$-minor-free graph $G$ is
\begin{enumerate}[label = \textnormal{(\alph*)}]
    \item isomorphic to a subgraph of $H \boxtimes K_{\floor{m}}$, where $\tw(H) \leq t - 1$ and $m \coloneqq \sqrt{(t - 3)n}$\textnormal{;}
    \item isomorphic to a subgraph of $H \boxtimes K_{\floor{m}}$, where $\tw(H) \leq t - 2$ and $m \coloneqq 2 \sqrt{(t - 3) n}$.
\end{enumerate}
\end{thm}

\Cref{KtMinorFreeSqrt}(a) immediately implies \cref{AST}, since 
\begin{equation*}
    \tw(G) \leq \tw(H \boxtimes K_{\floor{m}}) \leq (\tw(H) + 1)m - 1 < t \sqrt{(t - 3)n}.
\end{equation*}
The dependence on $n$ in the blow-up factor $m$ is best possible since the $n^{1/2} \times n^{1/2}$ planar grid graph $G$ is $K_5$-minor-free and has treewidth $n^{1/2}$. If $G$ is isomorphic to a subgraph of $H \boxtimes K_m$ where $H$ has bounded treewidth, then $n^{1/2} \leq \tw(G) \leq (\tw(H) + 1)m - 1$ and so $m = \Omega(n^{1/2})$. 

While our proof of \cref{KtMinorFreeSqrt} uses some ideas from the proof of \cref{AST} (in particular, \cref{FindTree} below), it is in fact significantly simpler, avoiding the use of havens or any form of treewidth duality. Instead, the proof directly constructs an isomorphism from $G$ to $H \boxtimes K_{\floor{m}}$ where $H$ is a graph obtained by repeated clique-sums (which implies the desired treewidth bound).

We also prove the following analogous theorem for excluded complete bipartite minors. Let \defn{$K^{\ast}_{s,t}$} be the graph whose vertex-set can be partitioned $A \cup B$, where $\abs{A} = s$, $\abs{B} = t$,  $A$ is a clique, and every vertex in $A$ is adjacent to every vertex in $B$, that is, $K^{\ast}_{s, t}$ is obtained from $K_{s, t}$ by adding all the edges inside the part of size $s$.

\begin{thm}
\label{KstMinorFreeSqrt}
For all integers $s, t\geq 2$, every $n$-vertex $K^{\ast}_{s, t}$-minor-free graph $G$ is isomorphic to a subgraph of $H \boxtimes K_{\floor{m}}$, where $\tw(H) \leq s$ and $m \coloneqq 2\sqrt{(s - 1)(t - 1)n}$.
\end{thm}

Again the $n^{1/2} \times n^{1/2}$ planar grid (which is $K_{3, 3}$-minor-free) shows the dependence on $n$ in the blow-up factor is best possible---we must have $m = \Omega(n^{1/2})$.

In light of \cref{AST}, it is natural to try to qualitatively strengthen \cref{KtMinorFreeSqrt,KstMinorFreeSqrt} by bounding the blow-up factor by a function of the treewidth of $G$, and ideally by a linear function of $\tw(G)$ since if $G \subseteq H \boxtimes K_m$ and $\tw(H) = \OO(1)$, then $m = \Omega(\tw(G))$. In this direction, \citet[Thm.~18]{UTW} proved that every $K_t$-minor-free graph $G$ is isomorphic to a subgraph of $H \boxtimes K_m$ where $\tw(H) \leq t - 2$ and $m  = \OO_t(\tw(G)^2)$. Similarly, they proved \citep[Thm.~19]{UTW} that every $K_{s,t}$-minor-free graph $G$ is isomorphic to a subgraph of $H \boxtimes K_m$  where $\tw(H) \leq s$ and $m = \OO_{s, t}(\tw(G)^2)$. Here $\OO_{s,t}(\cdot)$ and $\Omega_{s,t}(\cdot)$ hide dependence on $s$ and $t$. 

We achieve a blow-up factor that is linear in $\tw(G)$, and is independent of $t$ for $K_t$-minor-free graphs.

\begin{thm}
\label{KtMinorFreetw}
For any integer $t \geq 2$, every $K_t$-minor-free graph $G$ is isomorphic to a subgraph of $H \boxtimes K_m$, where $\tw(H) \leq t - 2$ and $m \coloneqq \tw(G) + 1$.
\end{thm}

The value of $m$ in \cref{KtMinorFreetw} is within a factor $t - 1$ of best possible, since
\begin{equation*}
    \tw(G) \leq \tw(H \boxtimes K_m) \leq (\tw(H) + 1)m - 1 < (t - 1)m.
\end{equation*}
Furthermore, the $t - 2$ bound on the treewidth of $H$ is best possible, since \citet[Thm.~18]{UTW} proved  that, for any function $f$ and for all $t$, there is a $K_t$-minor-free graph $G$ that is not a subgraph of $H \boxtimes K_{f(\tw(G))}$ for any graph $H$ with treewidth at most $t - 3$. 

For $K^{\ast}_{s, t}$-minor-free graphs we also obtain a blow-up factor that is linear in $\tw(G)$.

\begin{thm}
\label{KstMinorFreetw}
For all integers $s, t\geq 2$, every $K^{\ast}_{s, t}$-minor-free graph $G$ is isomorphic to a subgraph of $H \boxtimes K_m$, where $\tw(H) \leq s$ and $m \coloneqq (t - 1)(\tw(G) + 1)$.
\end{thm}

Here the value of $m$ is within a factor $(s + 1)(t - 1)$ of best possible and the $\tw(H) \leq s$ bound is best possible \citep[Thm.~19]{UTW}.

An attraction of \cref{KstMinorFreeSqrt,KstMinorFreetw} is that $\tw(H)$ depends on $s$ and not on the size of the excluded minor. This is particularly relevant for graphs of Euler genus\footnote{The \defn{Euler genus} of a surface with $h$ handles and $c$ cross-caps is $2h + c$. The \defn{Euler genus} of a graph $G$ is the minimum integer $g\geq 0$ such that $G$ embeds in a surface of Euler genus $g$; see \cite{MoharThom} for more about graph embeddings in surfaces.} $g$, since these contain no $K_{3, 2g + 3}$-minor. Thus the next result follow from \cref{KstMinorFreeSqrt,KstMinorFreetw}.

\begin{cor}
\label{Surface}
For any integer $g \geq 0$, every $n$-vertex graph $G$ of Euler genus $g$ is isomorphic to a subgraph of $H \boxtimes K_{\floor{m}}$, where $\tw(H) \leq 3$ and
\begin{equation*}
    m \coloneqq \min\set{4\sqrt{(g + 1)n},\, 2 (g + 1)(\tw(G) + 1)}.
\end{equation*}
\end{cor}

\Cref{Surface} is a product strengthening of results about balanced separators (equivalently, about treewidth) in graphs embeddable on surfaces of genus $g$, independently due to \citet{Djidjev81} and  \citet{GHT-JAlg84}. 
In particular, \cref{Surface} implies that $\tw(G) \leq (\tw(H)+1)m-1 = 4m -1 < 16\sqrt{(g + 1)n}$ and that $G$ has a balanced separator of size at most $4m\leq 16\sqrt{(g + 1)n}$. Both these bounds are tight up to the multiplicative constant. 

\Cref{KtMinorFreetw,KstMinorFreetw} are in fact special cases of a more general result, \cref{JJstMain}, that essentially converts any tree-decomposition of a graph excluding a particular minor into a strong product.  The starting tree-decomposition may be chosen to suit one's needs.  Making use of this flexibility, we deduce the Planar Graph Product Structure Theorem, \cref{PGPST}(b).

\begin{thm}[\cite{DJMMUW20}]\label{PGPST}
Every planar graph is isomorphic to a subgraph of\textnormal{:}
\begin{enumerate}[\textnormal{(\alph{*})}]
    \item $H \boxtimes P$ for some graph $H$ of treewidth  $8$ and for some path $P$.
    \item $H \boxtimes P \boxtimes K_3$ for some graph $H$ of treewidth $3$ and for some path $P$.
\end{enumerate}
\end{thm}

\Cref{PGPST} has been the key tool to resolve several open problems regarding queue layouts~\citep{DJMMUW20}, nonrepetitive colouring~\citep{DEJWW20}, $p$-centered colouring~\citep{DFMS21}, adjacency labelling~\citep{BGP20,EJM,DEJGMM21}, infinite graphs~\citep{HMSTW},  twin-width~\citep{BKW,BDHK}, and comparable box dimension~\citep{DGLTU22}. 

The bound of 3 on the treewidth of $H$ in (b) is tight \citep{DJMMUW20} even if $K_3$ is replaced by any constant-sized complete graph. Note that $\tw(H \boxtimes K_3) \leq 3 \tw(H) + 2$ for any graph $H$, so (b) implies (a) but with 8 replaced by 11. Our proof of \cref{PGPST}(b) removes much of the topology from the original proof, avoiding the use of Sperner's planar triangulation lemma. This allows us to prove a more general $H \boxtimes P \boxtimes K_m$ structure theorem, \cref{JJstLtw}, which we apply in the more general setting of apex-minor-free graphs, \cref{Apex}. This in turn has applications for $p$-centred colourings. 

\section{\large Preliminaries}\label{sec:prelim}

We consider simple finite undirected graphs $G$ with vertex-set $V(G)$ and edge-set $E(G)$. For each vertex $v\in V(G)$, let $N_G(v) = \set{w \in V(G) \colon vw\in E(G)}$. For $S \subseteq V(G)$, let $N_G(S) = \bigcup \set{N_G(v) \colon v \in S} \setminus S$.

A graph $H$ is a \defn{minor} of a graph $G$ if a graph isomorphic to $H$ can be obtained from a subgraph of $G$ by contracting edges. Say $G$ is \defn{$H$-minor-free} if $H$ is not a minor of $G$. A \defn{$K_r$-model} in a graph $G$ consists of pairwise-disjoint vertex-sets $(U_1, \dotsc, U_r)$ such that, for each $i$, the induced subgraph $G[U_i]$ is connected and, for all distinct $i, j$, there is an edge between $U_i$ and $U_j$. Clearly $K_r$ is a minor of a graph $G$ if and only if $G$ contains a $K_r$-model.

\subsection{Tree-decompositions and treewidth}\label{subsec:treedecom}

A \defn{tree-decomposition} $(T, \WW)$ of a graph $G$ consists of a collection $\WW=(W_x \colon x \in V(T))$ of subsets of $V(G)$, called \defn{bags}, indexed by the nodes of a tree $T$, such that:
\begin{itemize}
\item for each vertex $v \in V(G)$, the set $\set{x\in V(T) \colon v \in W_x}$ induces a non-empty (connected) subtree of $T$; and
\item for each edge $vw \in E(G)$, there is a node $x \in V(T)$ for which $v, w \in W_x$. 
\end{itemize}
The \defn{width} of such a tree-decomposition is $\max\set{\abs{W_x} \colon x \in V(T)} - 1$. The \defn{treewidth}, $\tw(G)$, of a graph $G$ is the minimum width of a tree-decomposition of $G$. Treewidth is the standard measure of how similar a graph is to a tree. Indeed, a connected graph has treewidth 1 if and only if it is a tree. Treewidth is of fundamental importance in structural and algorithmic graph theory; see \citep{Reed03,HW17,Bodlaender-TCS98} for surveys. 

We use the following property to prove treewidth upper bounds. A graph $G$ is a \defn{clique-sum} of graphs $G_1$ and $G_2$, if for some clique $\set{v_1, \dotsc, v_k}$ in $G_1$ and for some clique $\set{w_1, \dotsc, w_k}$ in $G_2$, $G$ is obtained from the disjoint union of $G_1$ and $G_2$ by identifying $v_i$ and $w_i$ for each $i$. In this case, it is well known and easily seen that $\tw(G) = \max\set{\tw(G_1), \tw(G_2)}$.

\subsection{Partitions}\label{subsec:part}

Instead of working with products, it is convenient to present our proofs using the following definition. A \defn{partition} of a graph $G$ is a graph $H$ such that:
\begin{itemize}
	\item each vertex of $H$ is a set of vertices of $G$, 
	\item each vertex of $G$ is in exactly one vertex of $H$, and
	\item for each edge $vw$ of $G$, if $v\in X\in V(H)$ and $w\in Y\in V(H)$ then $XY\in E(H)$ or $X=Y$. 
\end{itemize}
We call the vertices of $H$ the \defn{parts} of the partition. The \defn{width} of a partition is the size of its largest part. The \defn{treewidth} of a partition $H$ is $\tw(H)$. The next observation follows from the definitions and gives a useful characterisation of when a graph is isomorphic to a subgraph of a product of the form $H \boxtimes K_m$.

\begin{obs}
\label{ProductPartition}
A graph $G$ has a partition $H$ of width at most $m$ if and only if $G$ is isomorphic to a subgraph of $H \boxtimes K_{\floor{m}}$.
\end{obs}

In light of \cref{ProductPartition}, to prove our results it suffices to find a suitable partition. The following definition enables inductive proofs. A partition $H$ of a graph $G$ is \defn{rooted} at a $K_r$-model $(U_1, \dotsc, U_r)$ in $G$ if $U_1, \dotsc, U_r$ are vertices of $H$. Note that $U_1, \dotsc, U_r$ must be the vertices of an $r$-clique in $H$.

Finally, it will be useful to measure the `complexity' of a vertex-set with respect to a tree-decomposition $(T, \WW)$ of $G$. For a vertex-set $S \subseteq V(G)$, the \defn{$\WW$-width} of $S$ is the minimum number of bags of $\WW$ whose union contains $S$. The \defn{$\WW$-width} of a collection of vertex-sets is the maximum $\WW$-width of one of its sets. In a slight abuse of terminology, the \defn{$\WW$-width} of a partition $H$ of $G$ is the maximum $\WW$-width of one of the vertices of $H$.

\subsection{Hitting sets}\label{subsec:hitting}

Our proofs use results that say a collection of connected subgraphs of a graph (satisfying certain conditions) either has a small `hitting set' (a small set of vertices that meets every subgraph in the collection) or contains some suitable graphs. The following lemma is folklore (see \citep[(8.7)]{RS-V}). We include the proof for completeness. The  \defn{independence number $\alpha(G)$} of a graph $G$ is the size of a largest set $S\subseteq V(G)$ such that no edge of $G$ has both its end-vertices in $S$.

\begin{lem}
\label{TreeHit}
For any integer $\ell\geq 0$ and any collection $\FF$ of subtrees of a tree $T$, either\textnormal{:}
\begin{enumerate}[\textnormal{(}a\textnormal{)}]
    \item there are $\ell + 1$ vertex-disjoint trees in $\FF$, or
    \item there is set $S$ of at most $\ell$ vertices such that $S \cap V(T') \neq \emptyset$ for all $T' \in \FF$.
\end{enumerate}
\end{lem}

\begin{proof}
Let $I$ be the intersection graph of $\FF$. Since $T$ is a tree, $I$ is chordal and thus perfect. If $\alpha(I) \geq \ell + 1$, then (a) occurs. Otherwise $\alpha(I) \leq \ell$. Since $I$ is perfect, it has a partition $X_1, \dotsc, X_{r}$ into cliques where $r \leq \ell$. For each $i$, the subtrees in $X_i$ are pairwise intersecting. By the Helly property, there is a node $x_i \in V(T)$ in every subtree in $X_i$. Then $S \coloneqq \set{x_1, \dotsc, x_r}$ meets every subtree in $\FF$.
\end{proof}

In the setting of $\OO(\sqrt{n})$ blow-ups we need the following hitting set lemma due to \citet*{AST90}. Let $\FF$ be the collection of connected subgraphs of $G$ that intersect all of $A_1, \dotsc, A_k$. \Cref{FindTree} says that $\FF$ either contains a small graph or has a small hitting set.

\begin{lem}[{\cite[(1.2)]{AST90}}] 
	\label{FindTree}
	Let $G$ be a graph, $A_1, \dotsc, A_k$ be non-empty subsets of $V(G)$, and $x \geq 1$ be a real. Then either\textnormal{:}
	\begin{enumerate}[\textnormal{(}a\textnormal{)}]
		\item there is a tree $X$ in $G$ with $\abs{V(X)} \leq x$ such that $V(X) \cap A_i \neq \emptyset$ for each $i$, or
		\item there is a set $Y$ of at most $(k - 1) \abs{V(G)}/x$ vertices such that no component of $G - Y$ intersects all of $A_1, \dotsc, A_k$. 
	\end{enumerate}
\end{lem}

The next result is a straightforward extension of \cref{FindTree}.

\begin{lem}
	\label{FindTrees}
	Let $G$ be a graph, $A_1, \dotsc, A_k$ be non-empty subsets of $V(G)$, $x \geq 1$ be a real, and $\ell \geq 1$ be an integer. Then either\textnormal{:}
	\begin{enumerate}[\textnormal{(}a\textnormal{)}]
		\item there are pairwise disjoint trees $X_1, \dotsc, X_\ell$ in $G$ with $\abs{V(X_j)} \leq x$ and such that $V(X_j) \cap A_i \neq \emptyset$ for each $i$ and $j$, or
		\item there is a set $Y$ of at most $(\ell - 1)x + (k - 1)\abs{V(G)}/x$ vertices such that no component of $G - Y$ intersects all of $A_1, \dotsc, A_k$. 
	\end{enumerate}
\end{lem}

\begin{proof}
We proceed by induction on $\ell$. \Cref{FindTree} proves the result if $\ell = 1$. Now assume that $\ell \geq 2$ and the result holds for $\ell - 1$. 
If outcome (b) holds for $\ell - 1$, then the same set $Y$ satisfies outcome (b) for $\ell$. 
So assume that (a) holds for $\ell - 1$. That is, there are pairwise disjoint trees $X_1, \dotsc, X_{\ell - 1}$ in $G$ with $\abs{V(X_j)} \leq x$ and such that $V(X_j) \cap A_i \neq \emptyset$ for each $i$ and $j$.
Apply \cref{FindTree} to $G' \coloneqq  G - V(X_1 \cup \dotsb \cup X_{\ell-1})$. 
If there is a tree $X_\ell$ in $G'$ with $\abs{V(X_\ell)} \leq x$ such that $V(X_\ell) \cap A_i \neq \emptyset$ for each $i$, then $X_1, \dotsc, X_\ell$ are the desired set of trees, and outcome (a) holds. Otherwise there exists $Y' \subseteq V(G')$ with $\abs{Y'} \leq (k - 1) \abs{V(G)}/x$ such that no component of $G' - Y'$ intersects all of $A_1, \dotsc, A_k$. 
Let $Y \coloneqq  V(X_1 \cup \dots \cup X_{\ell-1}) \cup Y'$. 
Thus $\abs{Y} \leq (\ell - 1) x + (k - 1) \abs{V(G)}/x$ and no component of $G - Y$ intersects all of $A_1, \dotsc, A_k$ (since $G' - Y' = G - Y$). That is, $Y$ satisfies (b). 
\end{proof}

\section{\large\boldmath Main theorem and 
\texorpdfstring{$\OO(\tw(G))$}{O(tw(G)} blow-up}

We now prove our main technical theorem and deduce \cref{KtMinorFreetw,KstMinorFreetw} from it.

The following definition allows the $K_t$-minor-free and $K^{\ast}_{s, t}$-minor-free cases to be combined. Let \defn{$\JJ_{s, t}$} be the class of graphs $G$ whose vertex-set has a partition $A \cup B$, where $\abs{A} = s$ and $\abs{B} = t$, $A$ is a clique, every vertex in $A$ is adjacent to every vertex in $B$, and $G[B]$ is connected. A graph is \defn{$\JJ_{s,t}$-minor-free} if it contains no graph in $\JJ_{s,t}$ as a minor. The following is our main theorem.

\begin{thm}\label{JJstMain}
Let $s, t \geq 2$ be integers, $G$ be a $\JJ_{s, t}$-minor-free graph, and $(T, \WW)$ be a tree-decomposition of $G$. Then $G$ has a partition of $\WW$-width at most $t - 1$ and treewidth at most $s$.
\end{thm}

This says that, given a $\JJ_{s, t}$-minor-free $G$ and a tree-decomposition $(T, \WW)$ of $G$, there is a simple (low treewidth) partition that is also simple with respect to $\WW$.  \Cref{JJstMain} follows immediately from the next lemma (for example, by taking $r = 1$ and $U_1$ to consist of a single vertex).

\begin{lem}
\label{JJstMainRooted}
Let $s, t \geq 2$ be integers, $G$ be a $\JJ_{s, t}$-minor-free graph, and $(T, \WW)$ be a tree-decomposition of $G$. Suppose that $(U_1, \dotsc, U_r)$ is a $K_r$-model of $\WW$-width at most $t - 1$ where $r \leq s$. Then $G$ has a partition of $\WW$-width at most $t - 1$ and treewidth at most $s$ that is rooted at $(U_1, \dotsc, U_r)$. 
\end{lem}

\begin{proof}
Let $U \coloneqq U_1 \cup \dotsb \cup U_r$. 
We proceed by induction on $\abs{V(G)}$. If $V(G) = U$, then $(U_1, \dotsc, U_r)$ is the desired partition $H$ where $H = K_r$ has treewidth $r - 1 \leq s$. Now assume that $V(G) \setminus U \neq \emptyset$. Let $A_i \coloneqq N_{G}(U_i) \setminus U$ for each $i$.

First suppose that some $A_i$ is empty, say $A_1 = \emptyset$. By induction, $G - U_1$ has a partition $H_1$ of $\WW$-width at most $t - 1$ and treewidth at most $s$ that is rooted at $(U_2, \dotsc, U_r)$. Add a new part $U_1$ adjacent to each of $U_2, \dotsc, U_r$ to obtain the desired $H$-partition of $G$. The neighbourhood of $U_1$ is a clique on $r - 1$ vertices, so $\tw(H) = \max\set{\tw(H_1), r - 1} \leq s$. Thus we may assume that $A_i$ is non-empty for all $i$.

Next suppose that $G - U$ is disconnected. Then there is a partition $U, V_1, V_2$ of $V(G)$ into three non-empty sets such that there is no edge between $V_1$ and $V_2$. Let $G_1 \coloneqq G[U \cup V_1]$ and $G_2 \coloneqq G[U \cup V_2]$. For $j \in \set{1, 2}$, let $\WW_j$ be the tree-decomposition of $G_j$ obtained from $\WW$ by deleting all the vertices of $G$ not in $G_j$.  
By induction, each $G_j$ has a partition $H_j$ of $\WW_j$-width at most $t - 1$ and treewidth at most $s$ that is rooted at $(U_1, \dotsc, U_r)$. Let $H$ be the partition of $G$ obtained from $H_1$ and $H_2$ by identifying the vertex $U_i$ in $H_1$ with the vertex $U_i$ in $H_2$ for each $i$. The graph $H$ is a clique-sum of $H_1$ and $H_2$, so $\tw(H) = \max\set{\tw(H_1), \tw(H_2)} \leq s$. Since every bag of $\WW_1$ and $\WW_2$ is a subset of a bag of $\WW$, the partition $H$ has $\WW$-width at most $t - 1$. Thus we may assume that $G - U$ is connected.

We now show there exists a set $Y \subseteq V(G) \setminus U$ of $\WW$-width at most $t - 1$ such that 
\begin{equation}\label{eq:Ymain}
\textnormal{ no component of } G - U - Y  \textnormal{ meets every } A_i. \tag{$\dagger$}
\end{equation}
Let $\FF$ be the collection of all connected subgraphs $F$ of $G - U$ such that $V(F) \cap A_i \neq \emptyset$ for all $i$. For each $F\in \FF$, let $T_F \coloneqq T[\set{x \in V(T) \colon W_x \cap V(F) \neq \emptyset}]$. Since $F$ is connected, each $T_F$ is a (connected) subtree of $T$. 

First consider the case $r \leq s - 1$.

First suppose there exists $F_1, F_2 \in \FF$ such that $T_{F_1}$ and $T_{F_2}$ are disjoint. Let $xy$ be any edge of $T$ on the shortest path between $T_{F_1}$ and $T_{F_2}$. Then $W_x \cap W_y$ separates\footnote{Given a graph $G$ and $V_1, V_2 \subseteq V(G)$, a set $S$ \defn{separates} $V_1$ and $V_2$ if no connected component of $G - S$ contains a vertex of both $V_1$ and $V_2$.} $V(F_1)$ and $V(F_2)$. Let $S$ be a minimal subset of $W_x \cap W_y$ that separates $V(F_1)$ and $V(F_2)$. By construction, $S$ has $\WW$-width 1, $S \cap V(F_1) = \emptyset$, and $S \cap V(F_2) = \emptyset$. Then there is a partition $S \cup V_1 \cup V_2$ of $V(G) \setminus U$ such that $V(F_1) \subseteq V_1$, $V(F_2) \subseteq V_2$ and there is no edge between $V_1$ and $V_2$. We now show that $G[S \cup V_1]$ and $G[S \cup V_2]$ are connected. Consider some $s \in S$. Since $S$ is minimal, there is a path from $s$ to $V(F_1)$ internally disjoint from $S \cup V(F_2)$. Since there is no edge between $V_1$ and $V_2$, this path must lie entirely inside $S \cup V_1$. Since $F_1$ is connected, between any two vertices of $S$ there is a path entirely inside $S \cup V_1$. Since $G - U$ is connected, there is a path from any vertex of $V_1$ to $S$ inside $S \cup V_1$. Hence $G[S \cup V_1]$ is connected. Similarly for $G[S \cup V_2]$. For $j \in \set{1, 2}$, let $G_j$ be the graph obtained from $G$ by contracting all of $S \cup V_j$ into a single vertex $v_j$. Each $G_j$ is a minor of $G$ and thus is $\JJ_{s, t}$-minor-free. Furthermore, since $V(F_j) \subseteq V_j$, $(U_1, \dotsc, U_r, \set{v_j})$ is a $K_{r + 1}$-model in $G_j$. Let $\WW_j$ be the tree-decomposition of $G_j$ obtained from $\WW$ by replacing every instance of a vertex in $S \cup V_j$ by $v_j$. By induction, each $G_j$ has a partition $H_j$ of $\WW_j$-width at most $t - 1$ and treewidth at most $s$ that is rooted at $(U_1, \dotsc, U_r, \set{v_j})$. Let $H$ be obtained from the disjoint union of $H_1$ and $H_2$ where the corresponding $U_i$ are identified and the vertices $v_1$ and $v_2$ from $H_1$ and $H_2$ are identified and replaced by $S$. If $X \subseteq V(G_j) \setminus \set{v_j}$ is a subset of a  bag of $\WW_j$, then $X$ is a subset of a bag of $\WW$. So if $X \subseteq V(G_j) \setminus \set{v_j}$ has $\WW_j$-width at most $t - 1$, then $X$ has $\WW$-width at most $t - 1$. Since $S$ also has $\WW$-width at most $t - 1$, the partition $H$ has $\WW$-width at most $t - 1$. The graph $H$ is a clique-sum of $H_1$ and $H_2$, so $\tw(H) \leq \max\set{\tw(H_1), \tw(H_2)} \leq s$ and the partition has all the required properties.

Now assume that $T_{F_1}$ and $T_{F_2}$ intersect for all $F_1, F_2 \in \FF$. By the Helly property, there is a node $x\in V(T)$ such that $x\in V(T_{F})$ for all $F\in\FF$. Let $Y \coloneqq W_x$. Then $Y$ has $\WW$-width 1 and intersects every $F \in \FF$. Thus $G - U - Y$ contains no graph of $\FF$ and so every component of $G - U - Y$ avoids some $A_i$. This $Y$ satisfies \eqref{eq:Ymain}.

Now consider the case $r = s$.

Suppose that $\FF$ contains $t$ vertex-disjoint graphs $F_1, \dotsc, F_t$. Since $G - U$ is connected, there is a partition $Q_1, \dotsc, Q_t$  of $V(G) \setminus U$ such that $V(F_i) \subseteq Q_i$ and $G[Q_i]$ is connected, for all $i$. Contract each $Q_i$ to a single vertex $q_i$ and each $U_i$ to a single vertex $u_i$ to get a graph $G'$ with vertex-set $\set{u_1, \dotsc, u_s, q_1, \dotsc, q_t}$. Since $G - U$ is connected, $G'[\set{q_1, \dotsc, q_t}]$ is connected and so $G' \in \JJ_{s, t}$, a contradiction. Hence, there are no $t$ vertex-disjoint graphs in $\FF$. 
For any $F_1, F_2 \in \FF$, if $T_{F_1}$ and $T_{F_2}$ are disjoint, then $F_1$ and $F_2$ are disjoint. So $\set{T_F \colon F \in \FF}$ contains no $t$ pairwise disjoint subtrees. Thus, by \cref{TreeHit}, there is a set $S \subseteq V(T)$ of size at most $t - 1$ that meets every $T_F$. Let $Y \coloneqq \bigcup_{x \in S} W_x$. Then $Y$ has $\WW$-width at most $t - 1$ and intersects every $F \in \FF$. This $Y$ satisfies \eqref{eq:Ymain}.

We have shown in all cases that there exists $Y \subseteq V(G) \setminus U$ satisfying \eqref{eq:Ymain}. Take a minimal such $Y$ and let $G_1, \dotsc, G_q$ be the components of $G - U - Y$. Consider each $G_j$ in turn. Let $Y_j$ be the set of vertices $w \in Y$ that have a neighbour in $G_j$. By \eqref{eq:Ymain}, there exists $i'$ such that $A_{i'} \cap V(G_j) = \emptyset$. Since $G - U$ is connected and $A_{i'}$ is non-empty, both $Y$ and $Y_j$ are non-empty. We claim that for each $w \in Y_j$ there is a path $P_w$ from $w$ to $A_{i'}$ that avoids $U \cup V(G_j)$. By the minimality of $Y$, some component $Q$ of $G - U - (Y \setminus \set{w})$ meets every $A_i$. Since $Y$ satisfies \eqref{eq:Ymain},  $w$ is a cut-vertex of $Q$. Also $w$ has a neighbour in $G_j$, so $G_j$ is a subgraph of $Q$ and, furthermore, $G_j$ is a component of $Q - w$. Since $Q$ meets every $A_{l}$, there is a path $P_{w}$ from $w$ to $A_{i'}$ inside $Q$. But $V(G_j)$ does not meet $A_{i'}$ and $G_j$ is a component of $Q - w$, so $P_{w}$ avoids $G_j$. Also $P_{w}$ is in $Q$, so $P_{w}$ avoids $U$. Hence, $P_{w}$ has the required properties. Let $Z_j$ be the subgraph induced by the union of $U_{i'}$ and all $P_{w}$ (where $w \in Y_j$). By construction, $Z_j$ is connected and disjoint from $V(G_j) \cup (U \setminus U_{i'})$.

Take the subgraph of $G$ induced by $V(G_j) \cup Z_j \cup U$ and contract $Z_j$ into a new vertex $z_j$. Call the graph obtained $G'_j$, which has vertex-set $V(G_j) \cup (U \setminus U_{i'}) \cup \set{z_j}$. Now $(\set{z_j},U_i \colon i \neq i')$ is a $K_r$-model in $G'_j$. Let $\WW_j$ be the tree-decomposition of $G'_j$ obtained from $\WW$ by deleting vertices of $G$ not in $V(G_j)\cup Z_j\cup U$, and then replacing each vertex in $Z_j$ by $z_j$. By induction, $G'_j$ has a partition $H_j$ of $\WW_j$-width at most $t - 1$ and treewidth at most $s$ that is rooted at $(\set{z_j}, U_i \colon i \neq i')$. Add to $H_j$ the vertex $U_{i'}$ adjacent to all other $U_{i}$ and to $\set{z_j}$. Since the neighbourhood of this added vertex is a clique of order $r \leq s$, $H_j$ still has treewidth at most $s$.
Let $H$ be obtained from the disjoint union of $H_1, \dotsc, H_q$, where corresponding $U_i$ are identified and the vertices $z_1, \dotsc, z_q$ from $H_1, \dotsc, H_q$ are identified and replaced by $Y$. Note that if $X \subseteq V(G_j) \setminus \set{z_j}$ is a subset of a  bag of $\WW_j$, then $X$ is a subset of a bag of $\WW$. So if $X \subseteq V(G_j) \setminus \set{z_j}$ has $\WW_j$-width at most $t - 1$, then $X$ has $\WW$-width at most $t - 1$. Since $Y$ has $\WW$-width at most $t - 1$, the partition $H$ has $\WW$-width at most $t - 1$. The graph $H$ is a clique-sum of $H_1, \dotsc, H_q$,  so $\tw(H) \leq \max_j\tw(H_j) \leq s$. 

We finally check that $H$ is a partition of $G$. The vertices $U_1, \dotsc, U_r, Y$ form a clique in $H$ so all edges of $G$ inside $Y \cup U$ appear in $H$. Every edge inside $G_j$ appears in $G'_j - z_j$, thus appears in $H_j$ and hence in $H$. Any edge between $U$ and $G_j$ is, by definition of $i'$, an edge between $G_j$ and $U \setminus U_{i'}$ so appears in $G'_j - z_j$ and hence in $H$. Finally consider edges between $Y$ and $G_j$. Let $vw$ be an edge with $v \in V(G_j)$ and $w \in Y$. By definition, $w \in Y_j$ and so the edge $vz_j$ is present in $G'_j$ and hence in $H_j$. Since $z_j$ is replaced by $Y$, the edge $vw$ is in $H$.
\end{proof}

Applying \cref{JJstMain} to a tree-decomposition of minimum width gives the following corollary.

\begin{thm}
\label{JJstTw}
For all integers $s, t \geq 2$, every $\JJ_{s, t}$-minor-free graph $G$ is isomorphic to a subgraph of $H \boxtimes K_m$, where $\tw(H) \leq s$ and $m \coloneqq (\tw(G) + 1)(t - 1)$.
\end{thm}

\begin{proof}
Let $G$ be a $\JJ_{s, t}$-minor-free graph. Fix a tree-decomposition $(T, \WW)$ of $G$ in which every bag has size at most $\tw(G) + 1$. By \cref{JJstMain}, $G$ has a partition $H$ of $\WW$-width at most $t - 1$ where $\tw(H) \leq s$. Since each bag of $\WW$ has size at most $\tw(G) + 1$, the partition has width at most $(t - 1)(\tw(G) + 1) = m$. Hence, by \cref{ProductPartition}, $G$ is isomorphic to a subgraph of $H \boxtimes K_m$.
\end{proof}

Observe that $\JJ_{t - 2, 2} = \set{K_t}$ so every $K_t$-minor-free graph is $\JJ_{t - 2, 2}$-minor-free. Hence \cref{JJstTw} implies \cref{KtMinorFreetw}. Clearly, $K^{\ast}_{s, t}$ is a subgraph of every graph in $\JJ_{s, t}$ and so every $K^{\ast}_{s, t}$-minor-free graph is $\JJ_{s, t}$-minor-free. Hence, \cref{JJstTw} implies \cref{KstMinorFreetw}.

\section{\large Layered treewidth: planar and apex-minor-free graphs}
\label{LTW}

A \defn{layering} of a graph $G$ is a partition $\LL = (V_1, V_2, \dotsc)$ of $V(G)$ such that for each edge $vw \in E(G)$, if $v \in V_i$ and $w \in V_j$, then $\abs{i - j} \leq 1$.  A layering of $G$ is equivalent to a partition $P$ of $G$ where $P$ is a path. The next observation, first noted in \citep{DJMMUW20}, gives a useful characterisation of when a graph is isomorphic to a subgraph of a product of the form $H \boxtimes P \boxtimes K_m$.

\begin{obs}[\citep{DJMMUW20}]
\label{ProductLayering}
A graph $G$ has a layering $\LL$ and a partition $H$ such that each layer of $\LL$ and each part of $H$ intersect in at most $m$ vertices if and only if $G$ is isomorphic to a subgraph of $H \boxtimes P \boxtimes K_m$ for some path $P$.
\end{obs}

\begin{proof}
Suppose that $G$ is isomorphic to a subgraph of $H \boxtimes P \boxtimes K_m$ where $V(H) = \set{x_1, \dotsc, x_h}$, $V(P) = \set{y_1, y_2, \dotsc}$, and $V(K_m) = \set{z_1, \dotsc, z_m}$. Then the isomorphism maps each vertex $v$ of $G$ to $(x_{a(v)}, y_{b(v)}, z_{c(v)})$ where $v \mapsto (a(v), b(v), c(v))$ is injective. Let $\LL$ have layers $V_i = \set{v \colon b(v) = i}$ and the partition $H$ have parts $\set{v \colon a(v) = j}$ for $j \in \set{1, \dotsc, h}$. Since $c(v)$ takes at most $m$ values, each layer and part have at most $m$ vertices in common.

Reversing this identification converts a suitable layering $\LL$ and partition $H$ into an isomorphism from $G$ to a subgraph of $H \boxtimes P \boxtimes K_m$.
\end{proof}

\Citet{DMW17} defined the \defn{layered treewidth}, $\ltw(G)$, of $G$ to be the minimum integer $k$ such that $G$ has a layering $\LL$ and tree-decomposition $(T, \WW)$ such that $\abs{L \cap W} \leq k$ for each layer $L \in \LL$ and each bag $W \in \WW$.  \Cref{JJstMain} has the following corollary.

\begin{thm}
\label{JJstLtw}
For all integers $s, t \geq 2$, every $\JJ_{s, t}$-minor-free graph $G$ is isomorphic to a subgraph of $H \boxtimes P \boxtimes K_m$, where $P$ is a path, $\tw(H) \leq s$, and $m \coloneqq (t - 1) \ltw(G)$.
\end{thm}

\begin{proof}
Let $G$ be a $\JJ_{s, t}$-minor-free graph. Fix a layering $\LL$ and tree-decomposition $(T, \WW)$ of $G$ such that $\abs{L \cap W} \leq \ltw(G)$ for every layer $L \in \LL$ and each bag $W \in \WW$. By \cref{JJstMain}, $G$ has a partition $H$ of $\WW$-width at most $t - 1$ where $\tw(H) \leq s$.

Let $X \in V(H)$ be a part and $L \in \LL$ be a layer. Since the partition has $\WW$-width at most $t - 1$, there are bags $W_1, \dotsc, W_{t - 1} \in \WW$ such that $X \subseteq \bigcup_{i = 1}^{t - 1} W_i$. Since $\abs{L \cap W_i} \leq \ltw(G)$ for each $i$, $\abs{X \cap L} \leq (t - 1) \ltw(G)$. The result now follows from \cref{ProductLayering}.
\end{proof}

Again, since $\JJ_{t - 2, 2} = \set{K_t}$ and $K^{\ast}_{s, t}$ is a subgraph of every graph in $\JJ_{s, t}$, \cref{JJstLtw} has the following corollaries.

\begin{thm}\label{KtMinorFreeLayered}
For any integer $t \geq 2$, every $K_t$-minor-free graph $G$ is isomorphic to a subgraph of $H \boxtimes P \boxtimes K_m$, where $P$ is a path, $\tw(H) \leq t - 2$, and $m \coloneqq \ltw(G)$.
\end{thm}

\begin{thm}\label{KstMinorFreeLayered}
For all integers $s, t \geq 2$, every $K^{\ast}_{s, t}$-minor-free graph $G$ is isomorphic to a subgraph of $H \boxtimes P \boxtimes K_m$, where $P$ is a path, $\tw(H) \leq s$, and $m \coloneqq (t - 1) \ltw(G)$.
\end{thm}

The Planar Graph Product Structure Theorem (\cref{PGPST}(b)) follows from \cref{KtMinorFreeLayered} (with $t = 5$) and the fact that every planar graph has layered treewidth at most 3, as proved by \citet{DMW17}. We sketch the proof for completeness.

\begin{thm}[{\citep[Thm.~12]{DMW17}}]\label{planarltw}
Every planar graph has layered treewidth at most $3$.
\end{thm}

\begin{proof}[Proof Sketch.] 
We may assume that $G$ is a planar triangulation. Let $T$ be a breadth-first-search spanning tree rooted at an arbitrary vertex $r$. Let $G^{\ast}$ be the dual of $G$ and $T^{\ast}$ be the spanning subgraph of $G^{\ast}$ consisting of those edges not dual to edges in $T$.  \Citet{vonStaudt} showed that $T^{\ast}$ is a spanning tree of $G^{\ast}$.  For each vertex $x$ of $T^{\ast}$, corresponding to face $uvw$ of $G$, let $W_x$ be the union of the $ur$-path in $T$, the $vr$-path in $T$, and the $wr$-path in $T$.  \Citet{Eppstein99} showed that $(W_x \colon x \in V(T^{\ast}))$ is a tree-decomposition of $G$.  Let $V_i \coloneqq \set{v \in V(G) \colon \dist_G(v, r) = i}$ and so $(V_0, V_1, \dotsc)$ is a layering of $G$.  Since $T$ is a breadth-first-search spanning tree, each bag $W_x$ has at most three vertices in each layer $V_i$. Hence $\ltw(G) \leq 3$.
\end{proof}

We now show that the bound in \cref{planarltw} is tight. Suppose on the contrary that $\ltw(G) \leq 2$ for every planar graph $G$. Then each layer induces a subgraph with treewidth 1, which is thus a forest. Taking alternate layers, $G$ has a vertex-partition into two induced forests (which would imply the 4-colour theorem). \Citet{CK69} constructed planar graphs $G$ that have no vertex-partition into two induced forests, implying $\ltw(G) \geq 3$.

\Cref{PGPST} is generalised as follows. The \defn{vertex-cover number $\tau(G)$} of a graph $G$ is the size of a smallest set $S\subseteq V(G)$ such that every edge of $G$ has at least one end-vertex in $S$. 
By definition, $G$ is a subgraph of every graph in $\JJ_{\tau(G), \abs{V(G)}-\tau(G)}$. A graph $X$ is \defn{apex} if $X-v$ is planar for some vertex $v\in V(X)$. \Citet{DMW17} showed that for any graph $X$, the class of $X$-minor-free graphs has bounded layered treewidth if and only if $X$ is apex. Thus, the next result follows from \cref{KstMinorFreeLayered}.

\begin{thm}\label{Apex}
For every apex graph $X$ there exists $m \in \NN$, such that every $X$-minor-free graph is isomorphic to a subgraph of $H \boxtimes P \boxtimes K_m$, where $P$ is a path and $\tw(H) \leq \tau(X)$.
\end{thm}

\Citet{DJMMUW20} proved a similar result to \cref{Apex}, but with a much larger bound on $\tw(H)$ (depending on constants from the Graph Minor Structure Theorem).


\Cref{Apex} has applications to $p$-centred colouring, as we now explain. For $p \in \NN$, a vertex colouring of a graph $G$ is \defn{$p$-centred }if for every connected subgraph $X$ of $G$, $X$ receives more than $p$ colours or some vertex in $X$ receives a unique colour. The \defn{$p$-centred chromatic number $\chi_p(G)$} is the minimum number of colours in a $p$-centred colouring of $G$. Centred colourings are important within graph sparsity theory as they characterise graph classes with bounded expansion~\cite{Sparsity}. A result of \citet*[Lem.~8]{DFMS21} implies that $\chi_p(H \boxtimes P \boxtimes K_m) \leq m(p + 1) \chi_p(H)$ for every graph $H$. \Citet[Lem.~15]{PS21} proved that every graph of treewidth at most $t$ has $p$-centred chromatic number at most $\tbinom{p + t}{t} \leq (p + 1)^t$. In particular, \cref{Apex} implies:



\begin{thm}\label{pcentred}
For every apex graph $X$ with $\tau(X) \leq t$ there exists $m \in \NN$ such that for every $X$-minor-free graph $G$,
\begin{equation*}
    \chi_p(G)\leq m (p +1)^{t+1}.
\end{equation*}
\end{thm}

\citet{PS21} proved that for every graph $X$ there exists $c$ such that every $X$-minor-free graph has $p$-centred chromatic number $\OO(p^c)$. However, the known bounds on $c$ are huge (depending on the Graph Minor Structure Theorem). \Cref{pcentred} provides much improved bounds in the case of apex-minor-free graphs. As an example, since $K^{\ast}_{3,t}$ is apex with $\tau(K^{\ast}_{3,t})\leq 3$, \cref{pcentred} implies there exists $m = m(t)$ such that  $\chi_p(G)\leq m (p + 1)^4$ for every $K^{\ast}_{3,t}$-minor-free graph $G$.

\section{\large\boldmath Blow-up \texorpdfstring{$\OO(\sqrt{n})$}{O(rootn)}}

In this section we employ a similar proof strategy but with a different hitting result (\cref{FindTrees} in place of \cref{TreeHit}) to prove \cref{KtMinorFreeSqrt,KstMinorFreeSqrt}.

\begin{thm}\label{JJstSqrt}
Let $s, t, n$ be positive integers and define \begin{equation*}
m \coloneqq 
\begin{cases}
\max\set{t - 1, 1} & \textnormal{if } s = 1 \textnormal{ or } 2, \\
\sqrt{(s - 2)n} & \textnormal{if } s \geq 3 \textnormal{ and } t = 1,\\[1ex]
2\sqrt{(s - 1)(t - 1) n} & \textnormal{otherwise}.
\end{cases}
\end{equation*}
Then every $\JJ_{s,t}$-minor-free graph $G$ on $n$ vertices is isomorphic to a subgraph of $H \boxtimes K_{\floor{m}}$ for some graph $H$ of treewidth at most $s$.
\end{thm}

\Cref{JJstSqrt} implies \cref{KtMinorFreeSqrt,KstMinorFreeSqrt} since $\JJ_{t - 1, 1} = \JJ_{t - 2, 2} = \set{K_t}$ and $K^{\ast}_{s, t}$ is a subgraph of every graph in $\JJ_{s, t}$. \Cref{JJstSqrt} is implied by \cref{ProductPartition} and the following lemma.

\begin{lem}\label{JJstSqrtRooted}
Let $s, t, n$ be positive integers and define $m$ as in \cref{JJstSqrt}. Suppose $G$ is a $\JJ_{s, t}$-minor-free graph on $n$ vertices and $(U_1, \dotsc, U_r)$ is a $K_r$-model in $G$ where $r \leq s$ and $\abs{U_i} \leq m$ for all $i$. Then $G$ has a partition of width at most $m$ and treewidth at most $s$ that is rooted at $(U_1, \dotsc, U_r)$.
\end{lem}

\begin{proof}
Let $U \coloneqq U_1 \cup \dotsb \cup U_r$. 
We proceed by induction on $n$. If $n \leq r + m$, then the partition $(U_1, \dotsc, U_r, V(G) \setminus U)$ is the desired partition $H$ where $H = K_{r + 1}$ has treewidth $r \leq s$. Now assume that $n > r + m$. Note that if $n \leq t - 1$, then $n \leq m$ in all cases and so we may assume that $n > t - 1$. Let $A_i \coloneqq N_{G}(U_i) \setminus U$ for each $i$.

By the same argument used in the proof of \cref{JJstMainRooted}, we may assume that $A_i$ is non-empty for all $i$, and that $G - U$ is connected.

If $r \leq s - 1$ and there is some $U_{r + 1}$  of size at most $m$ such that $(U_1, \dotsc, U_{r + 1})$ is a $K_{r + 1}$-model in $G$, then \cref{JJstSqrtRooted} for $U_1, \dotsc, U_{r + 1}$ would imply it is also true for $U_1, \dotsc, U_r$ (with the same partition). In particular, if $r \leq s - 1$, then we may assume there is no $U_{r + 1}$ of size at most $m$ such that $(U_1, \dotsc, U_{r + 1})$ is a $K_{r + 1}$-model in $G$. Call this property the `maximality of $r$'.

We now show there exists a set $Y \subseteq V(G) \setminus U$ of size at most $m$ such that
\begin{equation}\label{eq:Y}
    \textnormal{no component of } G - U - Y \textnormal{ meets every } A_i. \tag{$\ddagger$}
\end{equation}

First suppose that $s = 1$ and so $U = U_1$. Suppose that $\abs{A_1} \geq t$. Let $v_1, \dotsc, v_t$ be distinct vertices in $A_1$. Since $G - U$ is connected, it is possible to partition $V(G) \setminus U$ into vertex-sets $Q_1, \dotsc, Q_t$ such that for all $i$, $v_i \in Q_i$ and $G[Q_i]$ is connected. Now contract each $Q_i$ into a single vertex $q_i$ and $U_1$ into a single vertex $u_1$ to get a graph $G'$ on vertex-set $\set{u_1, q_1, \dotsc, q_t}$. Since $G - U$ is connected, $G'[\set{q_1, \dotsc, q_t}]$ is connected and so $G' \in \JJ_{1, t}$, a contradiction. Hence $\abs{A_1} \leq t - 1 \leq m$. Then $Y = A_1$ satisfies \eqref{eq:Y}.

Next suppose that $s = 2$. If $r = 1$, then for any $x \in A_1$, the pair $(U_1, \set{x})$ is a $K_{2}$-model in $G$, which contradicts the maximality of $r$. Hence $r = 2$ and $U = U_1 \cup U_2$. Suppose $G - U$ contains $t$ pairwise vertex-disjoint paths $P_1, \dotsc, P_t$ from $A_1$ to $A_2$. Since $G - U$ is connected, there is a partition of $V(G) \setminus U$ into vertex-sets $Q_1, \dotsc, Q_t$ such that, for all $i$, $V(P_i) \subseteq Q_i$ and $G[Q_i]$ is connected. Now contract each $Q_i$ to a single vertex $q_i$ and each $U_i$ to a single vertex $u_i$ to get a graph $G'$ on vertex-set $\set{u_1, u_2, q_1, \dotsc, q_t}$. Since $G - U$ is connected, $G'[\set{q_1, \dotsc, q_t}]$ is connected and so $G' \in \JJ_{2, t}$, a contradiction. Thus, by Menger's theorem, there is a set $Y \subseteq V(G) \setminus U$ of size at most $t - 1 \leq m$ such that there is no path from $A_1$ to $A_2$ in $G - U - Y$. In particular, no component of $G - U - Y$ meets both $A_1$ and $A_2$ and so $Y$ satisfies \eqref{eq:Y}. Thus we may assume that $s \geq 3$.

Suppose that $r \leq s - 1$. Apply \cref{FindTree} to $G - U$ with $x = \sqrt{(s - 2) n} \geq 1$ and $k = r$. If (a) occurs, then there is a tree $T$ on at most $x \leq m$ vertices intersecting each $A_i$. Then $(U_1, \dotsc, U_r, T)$ is a $K_{r + 1}$-model in $G$ with all parts of size at most $m$, which contradicts the maximality of $r$. Hence, (b) occurs. That is, there is a vertex-set $Y$ of size at most $(r - 1)n/x \leq (s - 2)n/x = x \leq m$ such that no component of $G - U - Y$ meets every $A_i$. This $Y$ satisfies \eqref{eq:Y}.

Now assume that $r = s$. For $t = 1$ we are done: since $G - U$ is connected, contracting each of $U_1, \dotsc, U_s, G - U$ to a single vertex gives a $K_{s + 1}$-minor in $G$, which is a contradiction since $K_{s + 1} \in \JJ_{s, 1}$. Thus $t \geq 2$. Apply \cref{FindTrees} to $G - U$ with $\ell = t$, $k = r = s$ and $x = \sqrt{\frac{s - 1}{t - 1} n} > 1$. Suppose (a) occurs. Then there are pairwise disjoint trees $T_1, \dotsc, T_t$ in $G - U$ such that each $T_j$ meets each $A_i$. Since $G - U$ is connected, it is possible to partition $V(G) \setminus U$ into vertex-sets $Q_1, \dotsc, Q_t$ such that, for all $i$, $V(T_i) \subseteq Q_i$ and $G[Q_i]$ is connected. Now contract each $Q_i$ to a single vertex $q_i$ and each $U_i$ to a single vertex $u_i$ to get a graph $G'$ on vertex-set $\set{u_1, \dotsc, u_s, q_1, \dotsc, q_t}$. Since $G - U$ is connected, $G'[\set{q_1, \dotsc, q_t}]$ is connected and so $G' \in \JJ_{s, t}$, a contradiction. Hence, (b) occurs: there is a vertex-set $Y$ of size at most $(t - 1)x + (s - 1)n/x = m$ such that no component of $G - U - Y$ meets every $A_i$. This $Y$ satisfies \eqref{eq:Y}.

We have shown in all cases that there exists $Y \subseteq V(G) \setminus U$ satisfying \eqref{eq:Y}.  We may now finish exactly as in the proof of \cref{JJstMainRooted} (with width instead of $\WW$-width, so the argument is even simpler). 
\end{proof}

Since $K^{\ast}_{2, t}$ is planar and so $K^{\ast}_{2, t}$-minor-free graphs have bounded treewidth, one would expect a good bound (independent of $n$) on the blow-up factor.  \Citet{UTW} showed that every $K^{\ast}_{2, t}$-minor-free graph is isomorphic to a subgraph of $H \boxtimes K_{\OO(t^3)}$ where $\tw(H) \leq 2$. They state as an open problem whether this $\OO(t^3)$ bound can be improved to $\OO(t)$.  \Cref{JJstSqrt} for $s = 2$ gives an affirmative answer to this question.

\begin{thm}\label{K2tMinorFree}
For every integer $t \geq 2$, every $K^{\ast}_{2, t}$-minor-free graph $G$ is isomorphic to a subgraph of $H \boxtimes K_{t - 1}$, where $\tw(H) \leq 2$.
\end{thm}

Note that \cref{K2tMinorFree} implies $K^{\ast}_{2, t}$-minor-free graphs have treewidth $\OO(t)$, which was first proved by \citet[(4.4)]{LeafSeymour15}.

\section{\large Concluding Remarks}

We conclude the paper by first discussing some possible ways in which \cref{KtMinorFreeSqrt} might be strengthened. Similar questions can be asked for $K_{s,t}$-minor-free graphs. Consider the following meta-theorem: 

\smallskip

\quad 
\begin{minipage}{\textwidth-25mm}
Every $K_t$-minor-free graph $G$ is isomorphic to a subgraph of $H \boxtimes K_{m(G)}$\\ 
for some function $m$ and some graph $H$ of treewidth at most $f(t)$.
\end{minipage} 
\quad $(\star)$

\smallskip

Note that \cref{KtMinorFreeSqrt} says that $(\star)$ holds for $m(G) = 2 \sqrt{(t - 3)n}$ where $n \coloneqq \abs{V(G)}$ and $f(t) = t - 2$ while \cref{KtMinorFreetw} says it holds for $m(G) = \tw(G) + 1$ and $f(t) = t - 2$.

\textbf{Q1.} Is it possible to improve $f(t)$ in \cref{KtMinorFreeSqrt} (possibly sacrificing some dependence on $t$ in $m(G)$)? 
In particular, can $(\star)$ be proved with $m(G) = \OO_t(n^{1/2})$ and $f(t) = c$ for some constant $c$ independent of $t$?  It follows from a result of \citet{LMST08} that, even for planar graphs, $c \geq 2$. On the other hand, $(\star)$ holds with $H$ a star ($c = 1$) and $m(G) = \OO_t(n^{2/3})$, and for any $\epsilon > 0$ there exists $c$ such that $(\star)$ holds with $f(t) \leq c$ and $m(G) = \OO_t(n^{1/2 + \epsilon})$; see \citep{Wood22}. The real interest is when $m(G) = \OO_t(n^{1/2})$.

As noted in \cref{Intro}, there is no corresponding improvement to \cref{KtMinorFreetw} since $f(t) = t - 2$ is best possible when $m(G)$ is a function of $\tw(G)$.

\textbf{Q2.} We highlight the $t = 5$ case of Q1: is every $K_5$-minor-free graph $G$ isomorphic to a subgraph of $H \boxtimes K_{\OO(\sqrt{n})}$ for some graph $H$ of treewidth at most 2? Having treewidth at most 2 is equivalent to being $K_4$-minor-free, so this problem is particularly appealing. It is open even when $G$ is planar.

\textbf{Q3.} Optimising the dependence on $t$ in \cref{KtMinorFreeSqrt} is an interesting question. Note that \citet{KawaReed10} proved that $K_t$-minor-free graphs have balanced separators of order $\OO(t \sqrt{n})$, which is best possible. Does $(\star)$ hold with $f(t) \cdot m(G) = \OO(t \sqrt{n})$?

Finally we mention a connection to row treewidth. \Citet{BDJMW22} defined the \defn{row treewidth} of a graph $G$ to be the minimum treewidth of a graph $H$ such that $G$ is isomorphic to a subgraph of $H \boxtimes P$ for some path $P$. For example, \cref{PGPST}(a) says that planar graphs have row treewidth at most 8, which was improved to 6 by \citet*{UWY22}. It is easily seen that $\ltw(G) \leq \rtw(G) + 1$ for every graph $G$. The next result, which  provides a partial converse,  follows from \cref{KtMinorFreeLayered} since  $\tw(H \boxtimes K_m) \leq (\tw(H) + 1)m - 1$.

\begin{cor}\label{KtMinorFreeRow}
For every $K_t$-minor-free graph $G$,
\begin{equation*}
    \rtw(G) \leq (t - 1) \ltw(G) - 1.
\end{equation*}
\end{cor}

\Cref{KtMinorFreeRow} is in marked contrast to a result of \citet{BDJMW22} who constructed graphs with layered treewidth 1 and arbitrarily large row-treewidth. Thus the $K_t$-minor-free (or some other sparsity) assumption in \cref{KtMinorFreeRow} is necessary.

\textbf{Q4.} For what other graph classes $\GG$ (beyond those defined by an excluded minor) is row treewidth bounded by a function of layered treewidth for graphs in $\GG$?

\subsubsection*{Acknowledgement}

Thanks to Gwena\"{e}l Joret for pointing out that \cref{TreeHit} leads to improved dependence on $\ell$ in some of our results. Thanks to Robert Hickingbotham who observed that our results imply \cref{pcentred}. Thanks to Pat Morin for stimulating conversations. 

{
\fontsize{11pt}{12pt}
\selectfont
\bibliographystyle{DavidNatbibStyle}
\bibliography{DavidBibliography}
}
\appendix
\section{Simple Treewidth}
\label{SimpleTreewidth}


A tree-decomposition $(T,(W_x \colon x \in V(T)))$ of a graph $G$ is \defn{$k$-simple}, for some $k\in\NN$,  if it has  width  at most $k$, and for every set $S$ of $k$ vertices in $G$, we have $\abs{\set{x\in V(T) \colon S\subseteq W_x}}\leq 2$. The \defn{simple treewidth}, $\stw(G)$, of a graph $G$ is the minimum $k\in\NN$ such that $G$ has a $k$-simple tree-decomposition. Simple treewidth appears in several places in the literature under various guises \citep{BDJM,KU12,KV12,MJP06,Wulf16,HMSTW}. The following facts are well known: A graph has simple treewidth 1 if and only if every component is a path. A graph has simple treewidth at most 2 if and only if it is outerplanar. A graph has simple treewidth at most 3 if and only if it has treewidth at most 3 and is planar~\citep{KV12}. The edge-maximal graphs with simple treewidth 3 are ubiquitous objects, called  \defn{planar 3-trees} in structural graph theory and graph drawing~\citep{AP-SJADM96,KV12}, called \defn{stacked polytopes} in polytope theory~\citep{Chen16}, and called \defn{Apollonian networks} in enumerative and random graph theory~\citep{FT14}. It is well-known and easily proved that $\tw(G) \leq \stw(G)\leq \tw(G) + 1$ for every graph $G$ (see \citep{KU12,Wulf16}). 

Several known product structure theorems can be expressed in terms of simple treewidth. For example, the following simple treewidth version of \cref{PGPST} is known.
\begin{thm}\label{PGPSTstw}
Every planar graph is isomorphic to a subgraph of\textnormal{:}
\begin{enumerate}[\textnormal{(\alph{*})}]
    \item $H \boxtimes P$ for some planar graph $H$ of simple treewidth  $8$ and for some path $P$ \textup{(\citep{DJMMUW20})}.
    \item $H \boxtimes P$ for some planar graph $H$ of simple treewidth  $6$ and for some path $P$ \textup{(\citep{UWY22})}.
    \item $H \boxtimes P \boxtimes K_3$ for some planar graph $H$ of simple treewidth $3$ and for some path $P$ \textup{(\citep{DJMMUW20})}.
\end{enumerate}
\end{thm}

Similarly, this appendix shows that $\tw(H)$ can be replaced by $\stw(H)$ in \cref{KtMinorFreeSqrt}(a), \cref{KstMinorFreeSqrt}, \cref{KstMinorFreetw},
\cref{Surface}, \cref{KstMinorFreeLayered}, \cref{Apex}, and \cref{K2tMinorFree}. In particular, in \cref{K2tMinorFree}, $H$ is outerplanar and, in \cref{Surface}, $H$ is planar with treewidth at most 3.
These results all follow by improving the `treewidth at most $s$' conclusions of \cref{JJstMain,JJstSqrt} to `simple treewidth at most $s$'. This improvement comes at a slight cost: the theorems now apply to $K^{\ast}_{s,t}$-minor-free graphs instead of $\JJ_{s,t}$-minor-free graphs. The only place where we do not obtain a result in terms of simple treewidth is for $K_t$-minor-free graphs where we  use $\JJ_{t - 2, 2} = \set{K_t}$.

\begin{thm}\label{STWMain}
    Let $s, t \geq 2$ be integers, $G$ be a $K^{\ast}_{s,t}$-minor-free graph, and $(T, \WW)$ be a tree-decomposition of $G$. Then $G$ has a partition of $\WW$-width at most $t - 1$ and simple treewidth at most $s$.
\end{thm}

This result follows immediately from the next lemma, which is an analogue of \cref{JJstMainRooted} for simple treewidth. The main difference is we can no longer apply induction when $G - U$ is disconnected (pasting on the same clique can increase simple treewidth) and so we cannot assume $G - U$ is connected. The proof frequently uses the fact that for any clique $C$ in a graph $G$, any tree-decomposition of $G$ has a bag that contains $C$.

\begin{lem}\label{STWMainRooted}
    Let $s, t \geq 2$ be integers, $G$ be a $K^{\ast}_{s,t}$-minor-free graph, and $(T, \WW)$ be a tree-decomposition of $G$. Suppose that $(U_1, \dotsc, U_r)$ is a $K_r$-model of $\WW$-width at most $t - 1$ where $r \leq s$. Then $G$ has a partition $H$ of $\WW$-width at most $t - 1$ rooted at $(U_1, \dotsc, U_r)$ where $H$ has an $s$-simple tree-decomposition $(R, \BB)$. Furthermore, if $r = s$, then only one bag of $\BB$ contains all of $U_1, \dotsc, U_s$.
\end{lem}

\begin{proof}
Let $U \coloneqq U_1 \cup \dotsb \cup U_r$. 
We proceed by induction on $\abs{V(G)}$. If $V(G) = U$, then the $(U_1, \dotsc, U_r)$ is the desired partition $H$ where $H = K_r$, $R$ is a single vertex with bag $\set{U_1, \dotsc, U_r}$.
Now assume that $V(G) \setminus U \neq \emptyset$. Let $A_i \coloneqq N_{G}(U_i) \setminus U$ for each $i$.

First suppose that some $A_i$ is empty, say $A_1 = \emptyset$. By induction, $G - U_1$ has a partition $H_1$ of $\WW$-width at most $t - 1$ rooted at $(U_2, \dotsc, U_r)$ and $H_1$ has an $s$-simple tree-decomposition $(R_1, \BB_1)$. Add a new part $U_1$ adjacent to each of $U_2, \dotsc, U_r$ to get the partition $H$. Since $\set{U_2,U_3, \dots, U_r}$ is a clique in $H$, some bag $B_x \in \BB_1$ contains all of $U_2, \dotsc, U_r$. Add a leaf $y$ adjacent to $x$ and let $B_y \coloneqq \set{U_1, \dotsc, U_r}$. This gives the desired $s$-simple tree-decomposition $(R, \BB)$ of $H$. Thus we may assume that $A_i$ is non-empty for all $i$.

Next suppose that some component of $G - U$ does not meet every $A_i$. Without loss of generality some component $Q_1$ of $G-U$ misses $A_1$. Apply induction to $G_1 \coloneqq G[U \setminus U_1 \cup V(Q_1)]$ rooted at $(U_2, \dotsc, U_r)$ to obtain a suitable partition $H_1$ and $s$-simple tree-decomposition $(R_1, \BB_1)$. Apply induction to $G_2 \coloneqq G - V(Q_1)$ rooted at $(U_1, \dotsc, U_r)$ to obtain a suitable partition $H_2$ and $s$-simple tree-decomposition $(R_2, \BB_2)$. We obtain the partition $H$ for $G$ from the disjoint union of $H_1$ and $H_2$ where the corresponding $U_i$ ($2 \leq i \leq r$) are identified. Some bag $B_x \in \BB_1$ contains all of $U_2, \dotsc, U_r$ and some bag $B_y \in \BB_2$ contains all of $U_1, \dotsc, U_r$. Let $R$ be the tree obtained from the disjoint union of $R_1$ and $R_2$ with an edge added between $x$ and $y$. Then $(R, \BB_1 \cup \BB_2)$ is an $s$-simple tree-decomposition for $H$ (note that $V(G_1) \cap V(G_2) = U \setminus U_1$). Further, if $r = s$, then only one bag of $\BB_2$ contains all of $U_1, \dotsc, U_r$, and so only one bag of $\BB$ does. Now assume that every component of $G - U$ meets every $A_i$.

We now show there exists a set $Y \subseteq V(G) \setminus U$ of $\WW$-width at most $t - 1$ such that 
\begin{equation}\label{eq:YSTW}
\textnormal{ no component of } G - U - Y  \textnormal{ meets every } A_i. \tag{$\star$}
\end{equation}
Let $\FF$ be the collection of all connected subgraphs $F$ of $G - U$ such that $V(F) \cap A_i \neq \emptyset$ for all $i$. For each $F \in \FF$, let $T_F \coloneqq T[\set{x \in V(T) \colon W_x \cap V(F) \neq \emptyset}]$. Since $F$ is connected, each $T_F$ is a (connected) subtree of $T$. 

First consider the case $r \leq s - 1$.

First suppose there exists $F_1, F_2 \in \FF$ such that $T_{F_1}$ and $T_{F_2}$ are disjoint. Let $xy$ be any edge of $T$ on the shortest path between $T_{F_1}$ and $T_{F_2}$. Then $W_x \cap W_y$ separates $V(F_1)$ and $V(F_2)$. Let $S$ be a minimal subset of $W_x \cap W_y$ that separates $V(F_1)$ and $V(F_2)$.
By construction, $S$ has $\WW$-width 1, $S \cap V(F_1) = \emptyset$, and $S \cap V(F_2) = \emptyset$. Then there is a partition $S \cup V_1 \cup V_2$ of $V(G) \setminus U$ such that $V(F_1) \subseteq V_1$, $V(F_2) \subseteq V_2$ and there is no edge between $V_1$ and $V_2$. We now show that there is a component $Q_1$ of $G[S \cup V_1]$ that contains $S \cup V(F_1)$ and a component $Q_2$ of $G[S \cup V_2]$ that contains $S \cup V(F_2)$. Consider some $s \in S$. Since $S$ is minimal, there is a path from $s$ to $V(F_1)$ internally disjoint from $S \cup V(F_2)$. Since there is no edge between $V_1$ and $V_2$, this path must lie entirely inside $S \cup V_1$. Since $F_1$ is connected, the component of $G[S \cup V_1]$ containing $s$ contains all of $S \cup V(F_1)$. Similarly for $G[S \cup V_2]$. For $j \in \set{1, 2}$, let $G_j$ be the graph obtained from $G$ by contracting all of $Q_j$ into a single vertex $v_j$ and deleting the rest of $V_j$. Each $G_j$ is a minor of $G$ and thus is $K^{\ast}_{s, t}$-minor-free. Furthermore, since $V(F_j) \subseteq Q_j$, $(U_1, \dotsc, U_r, \set{v_j})$ is a $K_{r + 1}$-model in $G_j$. Apply induction to $G_j$ rooted at $(U_1, \dotsc, U_r, \set{v_j})$ to obtain a suitable partition $H_j$ and $s$-simple tree-decomposition $(R_j, \BB_j)$. Let $H$ be obtained from the disjoint union of $H_1$ and $H_2$ where the corresponding $U_i$ are identified and the vertices $v_1$ and $v_2$ from $H_1$ and $H_2$ are identified and replaced by $S$. There is a bag $B_x \in \BB_1$ and a bag $B_y \in \BB_2$ that both contain all of $U_1, \dotsc, U_r, S$. Let $R$ be the tree obtained from the disjoint union of $R_1$ and $R_2$ with an edge added between $x$ and $y$, and let $\BB \coloneqq \BB_1 \cup \BB_2$. If  $r=s-1$ then the bags $B_x$ and $B_y$ are unique and so only two bags in $\BB$ contain all of $U_1, \dotsc, U_r, S$. Thus $(R, \BB)$ is a $s$-simple tree-decomposition of $H$.

Now assume that $T_{F_1}$ and $T_{F_2}$ intersect for all $F_1, F_2 \in \FF$. By the Helly property, there is a node $x\in V(T)$ such that $x\in V(T_{F})$ for all $F\in\FF$. Let $Y \coloneqq W_x$. Then $Y$ has $\WW$-width 1 and intersects every $F \in \FF$. Thus $G - U - Y$ contains no graph of $\FF$ and so every component of $G - U - Y$ avoids some $A_i$. This $Y$ satisfies \eqref{eq:YSTW}.

Now consider the case $r = s$.

Suppose that $\FF$ contains $t$ vertex-disjoint graphs $F_1, \dotsc, F_t$. Contracting each $U_i$ and each $F_j$ to a vertex gives a $K^{\ast}_{s,t}$-minor in $G$. Hence, there are no $t$ vertex-disjoint graphs in $\FF$. 
For any $F_1, F_2 \in \FF$, if $T_{F_1}$ and $T_{F_2}$ are disjoint, then $F_1$ and $F_2$ are disjoint. So $\set{T_F \colon F \in \FF}$ contains no $t$ pairwise disjoint subtrees. Thus, by \cref{TreeHit}, there is a set $S \subseteq V(T)$ of size at most $t - 1$ that meets every $T_F$. Let $Y \coloneqq \bigcup_{x \in S} W_x$. Then $Y$ has $\WW$-width at most $t - 1$ and intersects every $F \in \FF$. This $Y$ satisfies \eqref{eq:YSTW}.

We have shown in all cases that there exists $Y \subseteq V(G) \setminus U$ satisfying \eqref{eq:YSTW}. Take a minimal such $Y$. Since $Y$ satisfies \eqref{eq:YSTW} there are induced subgraphs $G_1, \dotsc, G_r$ of $G - U - Y$ such that:
\begin{itemize}[noitemsep]
    \item each $G_j$ is a union of components of $G - U - Y$,
    \item $G_j$ does not meet $A_j$ for all $j$,
    \item every vertex of $G - U - Y$ is in exactly one $G_j$.
\end{itemize}
Let $Y_j$ be the set of vertices $w \in Y$ that have neighbours in $G_j$. We first show that if $G_j$ is non-empty, then so is $Y_j$. If not, then there is some $j$ for which $G_j$ is non-empty and there are no edges between $Y$ and $V(G_j)$. But then $G_j$ is a union of components in $G - U$. We showed above that every component of $G - U$ meets every $A_i$, so $A_j$ meets $G_j$, which is a contradiction.

We now only consider those $j$ with $G_j$ (and so $Y_j$) non-empty. We claim that for each $w \in Y_j$ there is a path $P_w$ from $w$ to $A_j$ that avoids $U \cup V(G_j)$. By the minimality of $Y$, some component $Q$ of $G - U - (Y \setminus \set{w})$ meets every $A_i$. Since $Y$ satisfies \eqref{eq:YSTW}, $w$ is a cut-vertex of $Q$. Now $Q$ meets $A_j$ and $w$ is adjacent to some vertex of $G_j$ so there is a path from $A_j$ to $V(G_j)$ in $Q$. There is no such path in $Q - w$, so there is some path $P_w$ from $A_j$ to $w$ in $Q$ that avoids $V(G_j)$. Also $P_{w}$ is in $Q$, so $P_{w}$ avoids $U$. Hence, $P_{w}$ has the required properties. Let $Z_j$ be the subgraph induced by the union of $U_{j}$ and all $P_{w}$ (where $w \in Y_j$). By construction, $Z_j$ is connected and disjoint from $V(G_j) \cup (U \setminus U_{j})$.

Take the subgraph of $G$ induced by $V(G_j) \cup Z_j \cup U$ and contract $Z_j$ into a new vertex $z_j$. Call the graph obtained $G'_j$, which has vertex-set $V(G_j) \cup (U \setminus U_{j}) \cup \set{z_j}$. Now $(\set{z_j},U_i \colon i \neq j)$ is a $K_r$-model in $G'_j$. Apply induction to $G'_j$ rooted at this $K_r$-model to obtain a suitable partition $H_j$ and $s$-simple tree-decomposition $(R_j, \BB_j)$. Let $H$ be obtained from the disjoint union of the $H_j$ where corresponding $U_i$ are identified, and the $z_j$ are identified and replaced by $Y$. This gives a partition of $G$ exactly as in the proof of \cref{JJstMainRooted}.

For each $j$, there is a bag $B_{x_j} \in \BB_j$ that contains all of $U_1, \dotsc, U_{j - 1}, U_{j + 1}, \dotsc, U_r, Y$. Let $R$ be the tree obtained from the disjoint union of the $R_j$ by adding a vertex $x$ adjacent to all the $x_j$. Let $B_x \coloneqq \set{U_1, \dotsc, U_r, Y}$ and define $\BB \coloneqq \set{B_x} \cup \bigcup_j \BB_j$. Since the only common neighbours of vertices in different $G_j$ are in $U \cup Y$, this is a tree-decomposition of $H$. If $r < s$, then, since each $(R_j, \BB_j)$ is $s$-simple, so is $(R, \BB)$. Finally, suppose that $r = s$. Consider $U_1, \dotsc, U_{j - 1}, U_{j + 1}, \dotsc, U_s, Y$: $B_{x_j}$ and $B_x$ are the only two bags of $\BB$ that contain all these sets. Finally, $B_x$ is the only bag containing all of $U_1, U_2, \dotsc, U_s$. In particular, $\BB$ is $s$-simple and satisfies the required properties.
\end{proof}

The next result is an analogue of \cref{JJstLtw} for simple treewidth, and is proved in the same way as \cref{JJstLtw}, using \cref{STWMain} instead of \cref{JJstMain}.

\begin{thm}
\label{KstLtw}
For all integers $s, t \geq 2$, every $K^{\ast}_{s, t}$-minor-free graph $G$ is isomorphic to a subgraph of $H \boxtimes P \boxtimes K_m$, where $P$ is a path, $\stw(H) \leq s$, and $m \coloneqq (t - 1) \ltw(G)$.
\end{thm}


Taking $s = 3$ in \cref{KstLtw} show that for all $t$ there is an $m$ such that every $K_{3,t}^{\ast}$-minor-free graph is isomorphic to a subgraph of $H \boxtimes P \boxtimes K_m$ where $P$ is a path and $H$ is planar with treewidth at most $3$. This has an application to $p$-centred colouring. \Citet[Thm.~6]{DFMS21} showed that if $H$ is planar with treewidth at most 3, then $\chi_p(H) = \OO(p^2 \log{p})$. Using this and $\chi_p(H \boxtimes P \boxtimes K_m) \leq m(p + 1)\chi_p(H)$ shows that every $K_{3, t}^{\ast}$-minor-free $G$ has $\chi_p(G) = \OO(p^3 \log p)$. This improves \cref{pcentred} which gives $\chi_p(G) = \OO(p^4)$.

Applying the same approach as in the proof of \cref{STWMainRooted} establishes the following analogue of \cref{JJstSqrt} for simple treewidth. We omit the details.

\begin{thm}
Let $s, t, n$ be positive integers and define
\begin{equation*}
    m \coloneqq \begin{cases}
    \max\set{t - 1, 1} & \textnormal{if } s = 1 \textnormal{ or } 2, \\
    \sqrt{(s - 2)n} & \textnormal{if } s \geq 3 \textnormal{ and } t = 1, \\
    2 \sqrt{(s - 1)(t - 1)n} & \textnormal{otherwise}.
    \end{cases}
\end{equation*}
Then every $K^{\ast}_{s, t}$-minor-free graph $G$ on $n$ vertices is isomorphic to a subgraph of $H \boxtimes K_{\floor{m}}$ for some graph $H$ of simple treewidth at most $s$.
\end{thm}

\end{document}